\documentclass[letterpaper, 10 pt, conference]{ieeeconf}  
\usepackage{amsmath}
\usepackage{amssymb}
\IEEEoverridecommandlockouts                              
\usepackage{graphics} 
\usepackage{epsfig} 
\usepackage{epstopdf}
\usepackage{mathptmx} 
\usepackage{times} 
\usepackage{amsmath} 
\usepackage{amssymb}  
\usepackage[linesnumbered,lined,boxed,commentsnumbered]{algorithm2e}
\usepackage[noend]{algpseudocode}
\usepackage{subfigure}
\usepackage{breqn}

\newcommand{\mathbbm}[1]{\text{\usefont{U}{bbm}{m}{n}#1}} 

\newtheorem{defn}{Definition}
\newtheorem{thm}{Theorem}
\newtheorem{prop}{Proposition}
\newtheorem{lem}{Lemma}
\newtheorem{rem}{Remark}
\newtheorem{cor}{Corollary}

\newtheorem{property}{Property}

\newcommand{\mcl}[1]{\mathcal{ #1}}
\newcommand{\mbb}[1]{\mathbb{ #1}}

\let\bbl\Bigl
\let\bbbl\biggl

\let\bbr\Bigr
\let\bbbr\biggr

\newcommand{\R}{\mathbb{R}}

\newcommand{\norm}[1]{\lVert{#1}\rVert}

\newcommand{\eps}{\varepsilon}

\title{\LARGE \bf
	A Dynamic Programming Approach to Evaluating Multivariate Gaussian Probabilities
}

\author{Morgan Jones%
	\thanks{M. Jones is with the School for the Engineering of Matter, Transport and Energy, Arizona State University, Tempe, AZ, 85298 USA. e-mail: {\tt \small morgan.c.jones@asu.edu } },
	Matthew M. Peet
	\thanks{M. Peet is with the School for the Engineering of Matter, Transport and Energy, Arizona State University, Tempe, AZ, 85298 USA. e-mail: {\tt \small mpeet@asu.edu } }
}

\begin{document}

		\maketitle
		\thispagestyle{empty}
		\pagestyle{empty}
\begin{abstract}
We propose a method of approximating multivariate Gaussian probabilities using dynamic programming. We show that solving the optimization problem associated with a class of discrete-time finite horizon Markov decision processes with non-Lipschitz cost functions is equivalent to integrating a Gaussian functions over polytopes. An approximation scheme for this class of MDP's is proposed and explicit error bounds under the supremum norm for the optimal cost to go functions are derived.
\end{abstract}

\section{Introduction}

Integration of a Gaussian function over a polytope is a central computational bottleneck in several control and optimization problems, including machine learning~\cite{Liao_2007}, chance constrained optimization~\cite{Hessem_2002,Hessem_2001,Prekopa_1978}, and statistical modeling~\cite{Bock_1996}. While this problem is known to be a computationally challenging problem~\cite{Cunningham_2011}, in this paper, we show that it can be reformulated as an optimization problem associated with a Markov Decision Process (MPD). Next, we show that MDPs of this form can be uniformly approximated by MDPs with countable state space. Finally, we show that this sequence of approximated MDPs can be efficiently solved using a variation of Belman's equation. The solution is then demonstrated in several numerical examples.


Many methods for integrating a Gaussian function over a polytope have emerged in the literature. Genz \cite{Genz_1992} represents the state of the art, where the algorithm makes a series of transformations to reduce integration over a hyper-rectangle to integration over a unit cube. Here lattice point numerical integration can be used and explicit error bounds can be achieved. However Genz only looks at the specific case where the integration is over a rectangle. Several other algorithms of Gaussian integration over rectangles can be found in~\cite{Gassmann_2002}. Another approach is to use bounding methods where the polytope is inner approximated by closed and bounded sectors such as in \cite{Hanebeck_2015}, but no error bound can be found here. A common approach is to use expectation propagation but as seen in \cite{Cunningham_2011} this method performs badly on anything that is more complicated than a rectangular integration region. An alternative method is to use probabilistic methods where confidence intervals can be provided instead of error bounds~\cite{Vijverberg_1997}. In this paper we propose an integration algorithm over a possibly non-compact general polytope with explicit error bounds.

In Section \ref{Section 2: Multi-variable normal integration over polytopes} we show the equivalence of solving the optimization problem associated with a class of MDP's and evaluating Gaussian probabilities. MDP's describe the mathematical framework for modeling discrete time evolving processes involving a decision making situation coupled with partly random outcomes. Each MDP has an associated optimization problem of picking the sequence of decisions that minimizes the total expected cost of the process. MDP's appear in a vast number of fields such as economics, computer science, engineering etc; an in depth list of application of MDP's can be found in the survey \cite{Aristotle_1993}.

MDP's are commonly solved using dynamic programing \cite{Bellman}. Unfortunately in practice it is rare to be able find an analytical solution to Bellman's equation and thus the problem must be solved numerically, see \cite{Jones_2017} as an example. In this paper we are interested in MDP's where the state and control spaces can be uncountable (for example $[0,1]$). In these cases for an algorithm to solve the problem it becomes necessary to approximate the MDP by discretization; that is we replace the state and control spaces with a countable set. One hopes there is sufficient continuity in the original MDP such that as the discretization sharpness increases a solution can be found arbitrarily close to the true solution.

In the literature there has been much work done on deriving error bounds for discretization approximations of MDP's with compact control and state spaces and Lipschitz cost functions~\cite{Bertsekas_1974,Hernandez_Lerma_1989}. However in many practical problems the state dynamics are of the form $x(t+1)=Ax(t) + B \epsilon(t)$ where $\epsilon \sim \mathcal{N}(0,1)$, inducing the non-compact state space of $x \in\mathbb{R}^n$. A major contribution was made in \cite{F_Dufour_2012} where a discretization scheme was proposed and error bounds were proved for a general class of MDP's with locally compact state and control spaces. In this paper we modify and extend the work of \cite{F_Dufour_2012} to the case when the terminal cost function of the MDP is non-Lipschitz. The discretization scheme we propose is to first approximate the cost function by a Lipschitz continuous function and then to use the discretization scheme from \cite{F_Dufour_2012}.

The rest of this paper is organized as follows. In section \ref{Section 2: Multi-variable normal integration over polytopes} we show the relation of MDP's and integrating Gaussian random variables over polytopes. In section \ref{section 3: Markov decision processes} we introduce the class of MDP's we are interested in approximating. In Section \ref{section 4: Approximating MDP's} we show how to approximate this class of MDP's. In Section \ref{Section 5: Numerical Results} we present our numerical results and in \ref{conclusion} we finish with our conclusion.

\section{Notation}
For a matrix $A \in \R^{m,n}$ we denote the j'th column of $A$ by $ A_{\cdot,j}=[a_{1,j}, \dots, a_{m,j}]^T$.

For $A \in \R^{m \times n}$ we define $||A||_{\text{max}}= \max \{ |a_{i,j}| \}$.


We define the power set of a set $S$ to be the set of all subsets denoted by $S^{\cup}:=\{U : U \subseteq S\}$.

For $T \in \mbb N$ we denote the set $[T]=\{0,...,T\}$.

We define the positive scalars as $\R^{++} :=\{x \in \R\,:\,x >0\}$.

For functions $f_1:X \to \mbb R$ and $f_2:X \to \mbb R$ we denote $f_1(x) \vee f_2(x) := \max\{f_1(x),f_2(x)\}$.

We denote the Hausdorff metric space in $\R^n$ as $D^n$ with metric $d_H$, which is the set of non-empty subsets of $\R^n$ where if $X,Y\in D^n$, then $d_H(X,Y)=\max\{\sup_{\{x \in X\}} \inf_{\{y \in Y\}} ||x -y||_2, \sup_{\{y \in Y\}} \inf_{\{x \in X\}} ||x-y||_2\}$.

The function $f: \mathbb{R}^n \to \mathbb{R}^m$ is said to be Lipschitz continuous if there exists $L>0$ such that :
		\begin{equation}
		||f(x_1)-f(x_2)||_2 \le L ||x_1-x_2||_2\qquad \text{for all } x_1,\, x_2 \in X\label{eqn:Lip}
		\end{equation}
For a Lipschitz continuous function $f:\mathbb{R}^n \to \mathbb{R}^m$, we denote by $L_{f}$ the smallest constant $L$ such that Equation~\eqref{eqn:Lip} holds.

For bounded function on $X$, we denote the infinity norm as $\norm{h}_\infty := \sup_{x \in X} |h(x)|$.

For a given weighting function $w: \R^n \rightarrow \R^{++}$, we also define the weighted infinity norm
$||v||_{w} := \sup_{x \in \R^n} \{ \frac{|v(x)|}{w(x)}\}$ and $v\in \mathbb{L}_w(X)$ to be the space of Lipschitz continuous functions with finite $||v||_{w}$.

	


We denote $\mcl B(X)$ to be the Borel sigma algebra of some set X.

Consider a probability space $(\Omega, \mcl F, \mbb P)$. We say $Z: \Omega \to \R$ is a real valued random variable if it is a $\mcl F$-measurable function. For any $B \in \mcl F$ we denote the law of $Z$ by $\mbb{P}_{Z}( B):= \mbb{P}(\{w : Z(w) \in B\})$. For a Borel measurable function $g: \R \to \R$ we define the expectation as $\mbb E_{Z}[g(Z)]:=\int_{\R} g d \mbb{P}_Z$. Furthermore we say $Z \sim \mathcal{N}(\mu,\Sigma)$, $\mu \in \mathbb{R}^n$ and $\Sigma \in \mathbb{R}^{n \times n}$ if $\mbb{P}_{Z}(B)= \int_B \phi(x) dx$ where \small{$\phi:\R^n \rightarrow [0,1]$ is given by $\phi({x})=\dfrac{1}{\sqrt{(2 \pi)^n \det(\Sigma)}} \exp\left(\dfrac{1}{2}({x} - \mu )^{T} \Sigma^{-1}({x} - \mu ) \right).$}\normalsize

For any subset $X \subset Z$, we define the indicator function $\mathbbm{1}_{X}: Z \rightarrow \{0,1\}$ as
\[
\mathbbm{1}_{X}(x)=\begin{cases}
                     1, & \mbox{if } x \in X \\
                     0, & \mbox{otherwise}.
                   \end{cases}
\]

In Section \ref{section 4: Approximating MDP's}, we will make use of a parameterized smoothed indicator function $g_{\lambda,b}(x):\R^m \rightarrow [0,1]$ which is defined for any $b\in \R^m$ and $\lambda>0$ as $g_{\lambda,b}(x):= \Pi_{i=1}^{m}g_{i}(x)$ where
\small{\begin{equation}\label{eqn: g}
g_{i}(x):= \begin{cases} 1, & \quad \text{if} \quad  x_i <b_i - \dfrac{1}{\lambda} \\
-\lambda(x_i - b_i), & \quad \text{if} \quad b_i - \dfrac{1}{\lambda} < x_i <b_i \\
0 , & \quad \text{if} \quad   b_i<x_i. \end{cases} \\ \quad
\end{equation} } \normalsize

Associated with $g_{\lambda,b}$, we define the region of smoothing $\mathcal{R} \subset \R^n$ as
	\begin{align*}
	\mathcal{R}_{\lambda,b} & =\{x \in \R^n : g_{\lambda,b}(x)\neq \mathbbm{1}_{\{x\in \R^n\;:\;x<b\}}(x)  \}
	\end{align*}

	Suppose $(X,D)$ is a compact metric space. We say the set $\Gamma_\beta=\{x_1,...,x_n\}$ is an $\beta$-partition of $X$ if:
	\begin{itemize}
		\item There exists disjoint subsets, $X_1,...,X_n$, of $X$ such that $\cup_{i=1}^{n}X_i=X$ and $x_i \in X_i$ for $i \in \{1,...,n\}$.
		\item $D(x,x_i) \le \beta$ for all $x \in X_i$.
	\end{itemize}
	Furthermore given a partition $\Gamma_\beta=\{x_1,...,x_n\}$ of some space $X$, we define $p_{X,\Gamma_\beta} : X \to \Gamma_\beta$ as $p_{X,\Gamma_\beta} (x)=x_i$ for every $x \in X_i$.

\section{Multi-variable Gaussian integration over polytopes can be written as a Dynamic Programing Problem} \label{Section 2: Multi-variable normal integration over polytopes}
Our aim is to compute:
\begin{equation} \label{eqn: polytope integral}
\mbb P_Z({Z} \in \mathcal{P})= \int_{ {x} \in \mathcal{P}}\phi( {x}) d {x}
\end{equation}
Where ${Z} \sim \mathcal{N}(\mu,\Sigma)$, $\mu \in \mathbb{R}^n$, $\Sigma \in \mathbb{R}^{n \times n}$ $\mathcal{P}=\{ {x} \in \mathbb{R}^n : A {x} \le b\}$, $A \in \mathbb{R}^{m \times n}$ and $b \in \mathbb{R}^m$.
\begin{rem}
	For any $\Sigma>0$, there exists an invertible $C \in \mathbb{R}^{n \times n}$ such that $ \Sigma = CC^T$ and under the transformation ${N}=C^{-1}({Z}-\mu)$ we see ${N} \sim \mathcal{N}( {0}, {I})$. Thus $\mbb P_Z( {Z} \in \mathcal{P})= \mbb P_N( {N} \in \mathcal{P}')$ where $\mathcal{P}'=\{ {x} \in \mathbb{R}^n : AC {x} \le b-A \mu\}=\{ {x} \in \mathbb{R}^n : A' {x} \le b'\}$. Therefore without loss of generality we can assume $\mu = {0}$ and $\Sigma ={I}$ for the rest of this paper.
\end{rem}
\begin{lem} \label{lem: Wlog last column is all non-zero}
	For every polytope $\mathcal{P} \subset \mathbb{R}^n$ there exists $A \in \R^{m \times n}$ and $b \in \mathbb{R}^m$ such that $\mathcal{P}=\{x \in \R^n : Ax \le b\}$ with $a_{i,n} \ne 0$ $\forall i \in \{1,\dots,m\}$.
\end{lem}
\begin{proof}
	Since $\mathcal{P}$ is a polytope there exists some $A' \in \R^{m \times n}$ and $b' \in \mathbb{R}^m$ such that $\mathcal{P}=\{x \in \R^n : A'x \le b'\}$. Now $A'x \le b'$ $\iff$ $TA'x \le Tb'$ for any elementwise-nonnegative invertible $T\in \R^{m \times m}$. WLOG we assume there is at least one nonnegative element in the last column of $A'$ (otherwise we can restrict the space to $\R^{n-1}$) and by relabeling coordinates we assume $a'_{m,n} \ne 0$. Consider the matrix \small{$T = \begin{bmatrix}
	& \frac{1}{2} & 0 & \dots & 0 & ||A'||_{\text{max}} / |a'_{m,n}| \\
	& 0     & \frac{1}{2} & \dots & 0 & ||A'||_{\text{max}}/ |a'_{m,n}|  \\
	& \vdots & \dots & \ddots & \vdots & \vdots \\
	& 0 & 0 & \dots & \frac{1}{2} & ||A'||_{\text{max}}/ |a'_{m,n}| \\
	& 0 & 0 & \dots & 0 & ||A'||_{\text{max}}/ |a'_{m,n}|
	\end{bmatrix}.$} \normalsize Clearly $T$ has all nonnegative elements and is invertible as all of its columns are independent. It follows $\mathcal{P}= \{ x \in \mathbb{R}^n : TA'x<Tb'\}$ where $TA'_{\cdot,n} = [\frac{a_{1,n}}{2} + sign(a_{m,n})||A'||_{\text{max}},...,\frac{a_{m-1,n}}{2} + sign(a_{m,n})||A'||_{\text{max}},sign(a_{m,n})||A'||_{\text{max}}  ]^{T}$ which clearly has no nonzero elements since $||A'||_{\text{max}} > a_{i,n}$ for $ 1 \le i \le m$.
	\end{proof}
We consider the Dynamic Programing (DP) problem:
\begin{align} \label{Simple_DP}
& J=\mathbb{E}_{x_{T}}\{\mathbbm{1}_{\{  {x} \le b\}}(x_T)\} \quad \text{Subject to:}\\ \nonumber
& x_{i,t+1}= x_{i,t} + a_{i,t+1} \epsilon_{t} \quad t\in [T-1] , \quad i \in \{1,..,m\}\\ \nonumber
& x_{i,0}=0 \quad i \in \{1,..,m\}\\ \nonumber
& \epsilon_{t} \sim \mathcal{N}(0,1) \quad t\in [T-1] \nonumber
\end{align}

\begin{prop} \label{prop: MDP = polytope integral}
	The objective function, $J$, defined in \eqref{Simple_DP} is equal to $\int_{{x} \in \mathcal{P}}\phi({x}) d{x}$, where  $\mathcal{P}=\{ {x} \in \mathbb{R}^T : A {x} \le b\}$ and $A=\{a_{i,t}\} \in \mathbb{R}^{m \times T}$.
\end{prop}
\begin{proof}
	Let us denote ${x_{t}}=(x_{1,t},...,x_{m,t})^T$. From the second line in \eqref{Simple_DP} we see ${x_{t+1}}={x_{t}} + {A}_{\cdot,t+1} \epsilon_{t}$. Thus:
	\small{\begin{align*}
	{x_{T}} & ={x_{T-1}}+{A}_{\cdot , T} \epsilon_{T-1}
	={x_{T-2}}+ {A}_{\cdot,T-1} \epsilon_{T-2} + {A}_{\cdot,T} \epsilon_{T-1}  \\
	& \cdots \cdots= {A}_{\cdot, 1} \epsilon_{0} +....+ {A}_{\cdot,T} \epsilon_{T-1}=A {z}
	\end{align*} } \normalsize
	Where ${z}=(\epsilon_{0},....,\epsilon_{T-1})$ and thus ${z} \sim \mathcal{N}( {0}, {I}_{T \times T})$.
	Now  considering the objective function in (2):
	\small{
	\begin{align*}
	J & = \mathbb{E}_{x_{T}}\mathbbm{1}_{\{  {x} \le b\}}(x_T) = \mathbb{E}_z\mathbbm{1}_{\{A {x} \le b\}}(z) = \mbb P_z(A {z} \le b)= \int_{{x} \in \mathcal{P}}\phi({x}) d{x}
	\end{align*}} \nonumber \end{proof}
Proposition \ref{prop: MDP = polytope integral} shows that computing integrals of Gaussian functions over polytopes is equivalent to solving a DP problem. Later we will discuss how to find an approximate solution to DP problems of the form \eqref{Simple_DP}.
\section{Markov decision processes} \label{section 3: Markov decision processes}
In this section state the properties of the class of MDP's we are interested in.
\subsection{Markov Decision Processes}
In this chapter we follow closely the notation and definitions of \cite{Lasserre_1999}.
\begin{defn} \label{def: MDP}
	We say $\mathcal{M}$ is a finite time horizon Markov Decision Process (MDP) if it is a six tuple $\mathcal{M}$ = $((\{X_t\}_{t \in \mathbb{N}},X),\mathcal{A},\psi, \{Q_t\}_{t\in \mathbb{N}}, (c,h),T)$ such that the following hold,
	\begin{itemize}
		\item $X$ is a locally compact Borel space, with metric $d_X$, representing the state space. $\{X_t\}_{t \in \mathbb{N}}$ is a family of locally compact Borel subsets of $X$ representing the state space at time $t$.
		
		\item $\mathcal{A}$ is a locally compact Borel space with metric $d_A$ representing the set of admissible inputs.
		
		\item $\psi$ is a map $X \to \mathcal{A}^{\cup}$ such that for each $x\in X$, $\psi(x)$ is a measurable subset of $\mathcal{A}$ representing the set of feasible controls that can be used at state $x \in X$. We suppose $\mathbb{K}_t = \{(x,a): x \in X_t, a \in \psi(x)\}$ and $\mathbb{K} = \{(x,a): x \in X, a \in \psi(x)\}$ are measurable subsets of $X \times \mathcal{A}^\cup$. (Note if $\mathcal{A}= \emptyset$ then we define $\mathbb{K}= X$ and $\mathbb{K}_t= X_t$).

		
		\item $\{Q_t\}_{t \in \mathbb{N}}$ is a family of stochastic kernels. That is,  for $B \in \mcl B (X_{t+1})$ the map $B \to Q(B|x,a)$ is a probability measure on $(X_{t+1}, \mcl B(X_{t+1}))$ for all $(x,a)\in \mbb K_t$, and $(x,a) \to Q(B|x,a)$ is a measurable function on $\mbb K$ for every $B \in \mcl B(X_{t+1})$. When $\mcl A=\emptyset$, we simplify our notation by $Q_t(B|x)=Q_t(B|x,a)$. We denote the Lebesgue integral $\int_B f(y)Q_t(dy|x,a):= \int_B f(y)d \mu$ where $\mu$ is the induced measure created by the stochastic kernel, $\mu(B)=Q_t(B|x,a)$.
%
		\item $c: \mathbb{K} \to \mathbb{R}$ is a measurable function representing the cost per stage and $h:X \to \mathbb{R}$ is a measurable function representing the terminal cost.
		
		\item $T \in \mathbb{N}$ with $T<\infty$, representing the terminal time step.
		

	\end{itemize}
	Furthermore we denote $\mathbb{M}$ to be the set of all finite time horizon MDP's.
\end{defn}

\begin{defn}
	Consider a MDP $\mathcal{M} = ((\{X_t\}_{t \in \mathbb{N}},X),\mathcal{A},\psi, \{Q_t\}_{t\in \mathbb{N}}, (c,h),T) \in \mathbb{M}$. We define a policy to be a sequence of maps $\pi=\{\pi_t\}_{t \in [T-1]}$ such that $\pi_t:X_t \to \mathcal{A}$ and for all ${t \in [T-1]}$ $\pi_t(x) \in \psi(x)$ $\forall x \in X_t$. We denote the space of policies for the MDP $\mathcal{M}$ by $\Pi_{\mathcal{M}}$.
\end{defn}


\begin{defn}
	For every MDP $\mathcal{M} = ((\{X_t\}_{t \in \mathbb{N}},X),\mathcal{A},\psi, \{Q_t\}_{t\in \mathbb{N}}, (c,h),T) \in \mathbb{M}$ we can define its associated optimization problem, $\mcl L_{\mcl M}(x_0)$. 
	
	\small
	\begin{align*}
	& \min_{\pi \in \Pi_{\mathcal{M}}} G_{\mathcal{M}}(x_0, \pi) := \mathbb{E}_x[ \sum_{t=0}^{T-1}c(x_t,\pi_t(x_t))+h(x_T)] \text{ Given,}\\
	& \mathbb{P}_{x_{t+1}}(x_{t+1} \in B | x_{t}=x, a=\pi_t(x_t))= Q_{t}(B|x,a), \hspace{0.1cm} B \in \mathcal{B}( X_{t+1}),\\
	& x(0)= x_0,
	\end{align*}
	\normalsize
	where $G_{\mathcal{M}}(x_0, \pi)$ denotes the expected cost for the policy $\pi \in \Pi_{\mathcal{M}}$ and initial condition $x_0 \in X_0$ associated with $\mathcal{M}$.
\end{defn}

\begin{defn}
	Consider a MDP, $\mathcal{M} \in \mathbb{M}$. The optimal total expected cost, $G^* : X_0 \to \R $ is defined by $G^*_{\mathcal{M}}(x) = \inf_{ \pi \in \Pi_{\mathcal{M}}} G_{\mathcal{M}}(x, \pi)$ for $x \in X_0$. We define $\pi^* \in \Pi_{\mathcal{M}}$ to be the optimal policy if $G_{\mathcal{M}}(x, \pi^*)=G^{*}_{\mathcal{M}}(x)$ for any $x\in X_0$.
\end{defn}

Commonly the associated optimization problem for an MDP is solved using a method called dynamic programing where Bellman's equation, which we will define in the next definition, is recursively solved backwards in time.
\begin{defn}
	For a MDP $\mathcal{M} \in \mathbb{M}$ we define the optimal cost to go function (OCTGF) $J_{\mathcal{M},t}:X \to \mathbb{R}$ recursively as:
	\begin{align} \label{bellman}
	J_{\mathcal{M},T}(x) & = h(x) \quad x \in X_T  \\ \nonumber
	J_{\mathcal{M},t}(x) & =\inf_{a \in \psi(x)}\{ c(x,a) + \int_{X_{t+1}} J_{\mathcal{M},t+1}(y)Q_{t}(dy|x,a) \}\\ \nonumber
	& \qquad \qquad  \quad x \in X_t, \quad t \in [T-1] \nonumber
	\end{align}
\end{defn}
\begin{prop}
	For any MDP $\mathcal{M} \in \mathbb{M}$, if $J_{\mathcal{M},t}(x)$ is the associated OCTGF and $G^*_{\mathcal{M}}(x)$ is the optimal expected cost, then $J_{\mathcal{M},0}(x) = G_{\mathcal{M}}^*(x)$ for all $x \in X_0$. Moreover, for every $t$, there exists $f_t:X\rightarrow \mcl A$ such that
\[	
J_{\mathcal{M},t}(x)  = c(x,f_t(x)) + \int_{X_{t+1}} J_{\mathcal{M},t+1}(y)Q_{t}(dy|x,f_t(x)) \}.
\]
$f_t$ then defines the optimal policy as $\pi^*=\{f_{t}\}_{t \in [T-1]}$.
\end{prop}
\begin{proof}
	See \cite{Lasserre_1999}.
	\end{proof}

\begin{cor} \label{cor: showing poltope evaluation problem is an MDP}
	There exists $\mcl M \in \mbb M$ such that the associated optimization problem, $\mcl L_{\mcl M}(0)$, is equivalent to \eqref{Simple_DP}.
\end{cor}
\begin{proof}
	We propose an MDP $\mathcal{M} = ((\{X_t\}_{t \in \mathbb{N}},X),\mathcal{A},\psi, \{Q_t\}_{t\in \mathbb{N}}, (c,h),T) \in \mathbb{M}$ with an associated optimization problem equivalent to \eqref{Simple_DP}. We define the elements of $\mathcal{M}$ as follows,
	
		\vspace{-0.25cm}
		 \begin{equation}  \label{1}
		 \hspace{-3.25cm} \bullet \quad X=\mathbb{R}^m \text{ and } X_t=X \text{ for all } t \in \mathbb{N}.
		\end{equation}
		\vspace{-0.5cm}
	\begin{equation} \label{2}
	\hspace{-7cm} \bullet \quad  \mathcal{A}= \emptyset.
	\end{equation}
		\begin{equation} \label{3}
	\hspace{-4.75cm} \bullet \quad \psi(x)= \emptyset \text{ for all } x\in X.
		\end{equation}
		 $\bullet$ We can define the family of stochastic kernels for $B \in \mathcal{B}( \mathbb{R}^m)$ and $x \in \R^m$,
		\vspace{-0.1cm}
			\begin{align} \label{4}
			Q_t(B|{x}) = \int_{-\infty}^{\infty} \mathbbm{1}_{B}({x} + A_{\cdot, t+1} y) \frac{1}{\sqrt{ 2 \pi}} \exp \left({\frac{-y^2}{2}} \right)dy
			\end{align}
			
			Where $A=[A_{\cdot,1},\cdots,A_{\cdot,T}] \in \mathbb{R}^{m \times T}$ as in \eqref{Simple_DP}.\\
		$\bullet$ The cost per stage and the terminal cost is,
		\begin{align} \label{5}
		& c( {x})=0, \quad h( {x})=\mathbbm{1}_{\{  {x} <b\}}(x).
		\end{align}
		Where $b \in \mathbb{R} ^m$ is from \eqref{Simple_DP}.\\
		 $\bullet$ The finite time horizon is given to be, 
\begin{equation} \label{6}
		 T < \infty \quad \text{is as defined in Eq.~\eqref{Simple_DP}}.
		 \end{equation} \end{proof}

\vspace{-0.4cm}
\subsection{Readily-aproximable MDP's}
Next we introduce similar properties of MDP's that \cite{F_Dufour_2012} approximates, however we allow for discontinuity in the terminal cost function and require Property \ref{ass: proability transitions to region tends to 0}.
\begin{defn}
	We say a six tuple $\mathcal{M}=((\{X_t\}_{t \in \mathbb{N}},X),\mathcal{A},\psi, \{Q_t\}_{t\in \mathbb{N}}, (c,h),T)$ is an \textbf{approximable MDP} or $\mcl M \in \mbb A \subset \mathbb{M}$ if $\mcl M$ satisfies the following Properties 1-8.
\end{defn}

%

\begin{property} \label{property: State space condition}
	$X=\R^m$.
\end{property}

\begin{property} \label{property: psi is compact for each x}
	$\psi(x)$ is compact for all $x \in X$.
\end{property}

\begin{property} \label{property: psi is lipshitz}
	The map $\psi: X \to \mathcal{A}^\cup$ is Lipschitz continuous with respect to the Hausdorff norm. So $d_H(\psi(x),\psi(y)) \le L_{\psi} d_X(x,y)$ for some constant $L_{\psi}>0$.
\end{property}


\begin{property} \label{ass: terminal}
	The cost function, $c: \mathbb{K} \to \mathbb{R}$, is Lipschitz continuous on $\mathbb{K}$. The terminal cost function, $h:X \to \mathbb{R}$, can be written in the form $h(x)=h_1(x) + h_2(x)$ where $h_1$ is a Lipschitz continuous and $h_2$ is of the form $h_2(x) = \begin{cases}
	f_1(x) & \quad x \le b \\
	f_2(x) & \quad x > b
	\end{cases}. $ Where $b \in X$ and $f_1$ and $f_2$ are bounded and Lipschitz continuous (we note $h_1$ is not necessarily bounded). Furthermore there exists a positive lower semi-continuous function $w: X \to \mathbb{R}$ and a positive constant $\bar{c}>0$ such that \begin{equation} \label{eqn: w}
	|h_1(x)|+||h_{2}||_{\infty} + \sup_{a \in \psi(x)} |c(x,a)|  < \bar{c} w(x) .
	\end{equation}
\end{property}
Before we proceed to Property 5 we will introduce some additional notation. Given a function $v:X \to \mathbb{R}$ and an MDP $\mathcal{M}=((\{X_t\}_{t \in \mathbb{N}},X),\mathcal{A},\psi, \{Q_t\}_{t\in \mathbb{N}}, (c,h),T) \in \mathbb{M}$, we define $\zeta_{v,t}^{\mathcal{M}}: \mathbb{K}_t \to \mathbb{R}$ by, \begin{equation} \label{eqn: zeta}
 \zeta_{v,t}^{\mathcal{M}}(x,a)=\int_{X_{t+1}} v(y)Q_{t}(dy|x,a).
\end{equation} We note for the MDP with tuple elements defined \eqref{1} to \eqref{6} we have $\zeta_{v,t}^{\mathcal{M}}(x)=\int_{\mathbb{R}} v(x+ {A}_{\cdot,t+1} w) \phi(w) dw$ for $x \in \mathbb{R}^m$.

\begin{property} \label{ass: zeta of w inequlity}
	There exists $w$ satisfying \eqref{eqn: w} such that $\zeta_{w,t}^{\mathcal{M}}(x,a)$ is upper continuous on $\mathbb{K}_t$. In addition there exists $\bar{d}>0$ such that $\zeta_{w,t}^{\mathcal{M}}(x,a) \le \bar{d} w(x)$ for all $(x,a) \in \mathbb{K}_t$.
\end{property}

\begin{property} \label{property: zeta is continuous}
	For every bounded and continuous function $v$ on $X$, the MDP $\mathcal{M}$ has the property that the induced function $\zeta_{v,t}^{\mathcal{M}}$ is continuous on $\mathbb{K}$ for each $t\in \mathbb{N}$ .
\end{property}

\begin{property} \label{property: Zeta is lipshitz}
	There exists a constant $L_q>0$ such that $\forall t \in \mathbb{N}$, $(x,a)$ and $(y,b)$ in $\mathbb{K}_t$ and for any Lipschitz continuous function $v:X \to \mathbb{R}$ with Lipschitz constant $L_v >0$:
	\begin{equation*}
	|\zeta_{v,t}^{\mathcal{M}}(x,a)-\zeta_{v,t}^{\mathcal{M}}(y,b)|<L_{q}L_{v}|d_X(x,y)+d_A(a,b)|
	\end{equation*}
\end{property}


\begin{property} \label{ass: proability transitions to region tends to 0}
		Consider $w: X \to \R$ and $b \in \R^m$ as in Property \ref{ass: terminal}, then for all $\theta >0$ there exists an $\Lambda \in \mathbb{R}$ such that,\\
		 $|sup_{x \in X,u \in \psi(x)} \int_{y \in \mathcal{R}_{\lambda,b}} w(y) Q_{T-1}(dy|x,u)|< \theta$ for all $\lambda>\Lambda$.
\end{property}
Next we will prove a Lemma showing the MDP associated with tuple elements defined \eqref{1} to \eqref{6} has Property \ref{ass: proability transitions to region tends to 0}. Then in the next proposition we will show the MDP is in $\mathbb{A}$.
	\begin{lem} \label{lem: Bound_of_transition_to_bad_region}
Consider the stochastic kernel, of the MDP with tuple elements \ref{1} to \ref{6},
\[
			Q_t(B|{x})  = \int_{-\infty}^{\infty} \mathbbm{1}_{B}({x} + A_{\cdot,t+1} y) \frac{1}{\sqrt{ 2 \pi}} \exp\left({\frac{-y^2}{2}}\right)dy
\]
Where $B \in \mcl B (\R^m)$, $b \in \R^m$, $x \in \R^m$ and $A \in \R^{m \times T}$,
 then for all $\lambda>0$ we have,
		\small{\begin{equation} \label{eqn: final bound on transitioning into the bad region}
		\left|\sup_{x \in \mathbb{R}^m} \int_{y \in \mathcal{R}_{\lambda,b}} Q_{T-1}(dy|x)\right|<\frac{m}{\min_{1 \le i \le m}\{|a_{i,T}|\}\lambda}.
		\end{equation}} \normalsize
	\end{lem}
	\begin{proof}
	By Lemma \ref{lem: Wlog last column is all non-zero} we have  $a_{i,T} \ne 0$ $\forall i \in \{1,..,m\}$. For some $x \in \mathbb{R}^m$ and $\lambda>0$ we have,	
	\vspace{-0.2cm}
		\begin{align} \label{eqn: initial inequality for tranistioning into the bad region}
		\int_{y \in \mathcal{R}_{\lambda,b}} Q & _{T-1}  (  dy|x)  = \mathbb{P}_{\eps_{T-1}}  (x + A_{\cdot,T} \epsilon_{T-1} \in \mathcal{R}_{\lambda,b} )\\ \nonumber
		& = \mathbb{P}_{\eps_{T-1}} \left(\cup_{1 \le i \le m} \left\{ x_i + a_{i,T} \epsilon_{T-1} \in (b_i-\dfrac{1}{\lambda},b_i) \right\} \right)\\ \nonumber
		& \le \sum_{i=1}^{m} \mathbb{P}_{\eps_{T-1}} \left(\epsilon_{T-1} \in \dfrac{1}{a_{i,T}}(b_{i}-\dfrac{1}{\lambda} - x_i ,b_{i} - x_i) \right).
		\end{align}
Where $\eps_{T-1} \sim \mcl N(0,1)$. For $i \in \{1,...,m\}$ let us consider the function $f_i: \mathbb{R} \to [0,1]$ defined by,
		\vspace{-0.4cm} \begin{align*}
		f_i(x) & =\mathbb{P}_{\eps_{T-1}}\left(\epsilon_{T-1} \in \dfrac{1}{a_{i,T}}(b_{i}-\dfrac{1}{\lambda} - x ,b_{i} - x) \right)\\
		&= \int_{\frac{b_{i}-\frac{1}{\lambda}-x}{a_{i,T}}}^{\frac{b_{i}-x}{a_{i,T}}} \frac{1}{\sqrt{2 \pi}} \exp\left(\dfrac{-w^2}{2} \right) dw.
		\end{align*}
It can be shown $x^*=b_{i} - \frac{1}{2\lambda}$ is the point at which $f_i$ attains its maximum.
Now,
		\begin{align*}
		f_i(x^*)= &  \int_{-\frac{1}{2a_{i,T}\lambda} }^{\frac{1}{2a_{i,T}\lambda} } \frac{1}{\sqrt{2 \pi}} \exp \left(\dfrac{-w^2}{2} \right) dw
		\le  \int_{-\frac{1}{2a_{i,T}\lambda} }^{\frac{1}{2a_{i,T}\lambda} } dw\\
		= &\frac{1}{|a_{i,T}|\lambda}
		\le  \frac{1}{\min_{1 \le i \le m}|a_{i,T}|\lambda}.
		\end{align*}
Now by substituting this into \eqref{eqn: initial inequality for tranistioning into the bad region} we derive \eqref{eqn: final bound on transitioning into the bad region}. \end{proof}

\begin{prop} \label{prop: proves our polytope MDP is readily-aproximable}
	Let us denote the MDP with tuple elements defined \eqref{1} to \eqref{6} by $\mathcal{M}$, then $\mathcal{M} \in \mathbb{A}$.
\end{prop}
\begin{proof}
	To show $\mathcal{M}\in \mathbb{A}$ we will show $\mathcal{M}$ satisfies Properties 1-8.\\
	\underline{Property \ref{property: State space condition}}: True since $X=\mathbb{R}^m$. \\
	 \underline{Properties \ref{property: psi is compact for each x} and \ref{property: psi is lipshitz}}: $\emptyset$ is compact and $\psi(x)=\emptyset$ $\forall x \in X$, moreover it follows $d_H(\psi(x), \psi(y))=0$ for all $x,y \in X$.\\
	\underline{Property \ref{ass: terminal}}: $c(x,a) \equiv 0$, $h_1(x) \equiv 0$ and $h_2(x) = \begin{cases}
	& 1 \quad  x \le b\\
	& 0 \quad x >b
	\end{cases}$. We can trivially select $w(x) \equiv 1$ in this case.\\
			\underline{Property \ref{ass: zeta of w inequlity}}: The probability measure of the entire state space is 1. $\zeta_{w,t}^{\mathcal{M}}(x,a) = \int_X Q_t(dy|x,a)=1 = w(x)$.\\
	\underline{Property \ref{property: zeta is continuous}}: Consider continuous and bounded function $v:\mathbb{R}^m \to \mathbb{R} $ and let $C= ||v||_{\infty}$. Let us denote $\phi(u)= \frac{1}{\sqrt{2 \pi}} \exp\left(\dfrac{-u^2}{2} \right)$. We can use Dominated Convergence Theorem (DCT) to show $\zeta_{v,t}^{\mathcal{M}}(x)= \int_{y \in \mathbb{R}^m} v(y) Q_{t}(dy|x)= \int_{-\infty}^{\infty} v(x + A_{t+1} u)\phi(u)du$ is continuous with respect to $x$. Suppose $ \lim_{n \to \infty}x_n = x$ and let $g_n(u) =  v(x_n + A_{t+1} u)\phi(u)$. Since $v$ is continuous clearly $\lim_{n \to \infty}g_n = g=v(x + A_{t+1} u)\phi(u)$. Now $g_n(u) \le \sup_x |v(x + A_{t+1} u)|\phi(u) \le C \phi(u)$. Thus $g_n(u)$ is dominated by some integrable function $C \phi(u)$ ($\int |C \phi(u)| du =C< \infty$) and tends point-wise to $g(u)$. It follows by DCT $\lim_{n \to \infty} \int g_n(u) du = \int g(u) du$, showing $\zeta_{v,t}^{\mathcal{M}}(x)$ is continuous.\\
	\underline{Property \ref{property: Zeta is lipshitz}}: We will show $L_q=1$. Suppose $v$ is a Lipschitz continuous function.
	\begin{align*}
	& \left|\zeta_{v,t}^{\mathcal{M}}(x,a)- \zeta_{v,t}^{\mathcal{M}}(y,b) \right| \\ & = \left|\int_{-\infty}^{\infty} [v(x + A_{\cdot, t+1} w)-v(y + A_{\cdot, t+1} w)] \phi(w)dw \right|\\
	& \le L_{v} \int_{-\infty}^{\infty} \left|(x + A_{\cdot,t+1} w)-(y + A_{\cdot,t+1} w) \right|\phi(w)dw\\
	& = L_{v} |x-y|
	\end{align*}
	\underline{Property \ref{ass: proability transitions to region tends to 0}}: The result follows from Lemma \ref{lem: Bound_of_transition_to_bad_region}. \end{proof}

\section{Approximating MDP's} \label{section 4: Approximating MDP's}

Given $\mcl M \in \mbb A$ our approximation scheme has two stages; smoothing and discretization. During the smoothing stage the terminal cost function of the MDP is approximated with a Lipschitz continuous function. During the discretization stage the state and control spaces are approximated with compact spaces and then further approximated to countable sets.

\subsection{Smoothing}
For any MDP $\mathcal{M} \in \mathbb{A}$ we will show how to use the function $g_{\lambda,b}(x)$ to construct a sequence of MDP's with smooth terminal cost function and OCTGF's that converge to the OCTGF of $\mcl M$ under the supremum norm. 



	

\begin{defn}
	Consider an approximable MDP $\mathcal{M}=((\{X_t\}_{t \in \mathbb{N}},X),\mathcal{A},\psi, \{Q_t\}_{t\in \mathbb{N}}, (c,h),T) \in \mathbb{A}$. By Property \ref{ass: terminal} we can write $h(x)=h_1(x)+h_2(x)$ where $h_1$ is Lipschitz continuous and $h_2(x) = \begin{cases}
	f_1(x) & \quad x \le b \\
	f_2(x) & \quad x > b
	\end{cases} $. Let us define the smooth function $\tilde{h}_{\lambda}(x;b; \mathcal{M})=f_1(x)g_{\lambda,b}(x) + f_2(x)(1-g_{\lambda,b}(x))$. We call the MDP $\tilde{\mathcal{M}_\lambda}=((\{X_t\}_{t \in \mathbb{N}},X),\mathcal{A},\psi, \{Q_t\}_{t\in \mathbb{N}}, (c,h_1 + \tilde{h}_\lambda),T)$ the $\lambda$-smoothed MDP of $\mathcal{M}$. Furthermore we define the map $\Phi_1 : \mathbb{A} \times \mathbb{R}^+ \to \mathbb{A} $ by $\Phi_1(\mathcal{M}, \lambda)=\tilde{\mathcal{M}}_\lambda$.
\end{defn}
Next we will show that the terminal cost function of the $\lambda$-smoothed MDP is Lipschitz continuous.
\begin{cor} \label{cor: h tilde Lipschitz.}
	The function $\tilde{h}_\lambda:X \to \mathbb{R}$ defined by $\tilde{h}_\lambda (x;b, \mathcal{M})=f_1(x)g_{\lambda,b}(x) + f_2(x)(1-g_{\lambda,b}(x))$, where $f_1$ and $f_2$ are any bounded Lipschitz functions, is Lipschitz continuous with Lipschitz constant $L_{\tilde{h_\lambda}}=[L_{f_{1}}+L_{f_{2}} + 2 \lambda m \max\{||f_1||_{\infty},||f_2||_{\infty}\}]$. Where $m$=dim($X$).
\end{cor}
\begin{rem} \label{rem: first map sends aproximable MDP to aproximable MDP}
	The image of the map $\Phi_1$ is a subset of $\mathbb{A}$. Furthermore for any $\mcl M \in \mbb A$ and $\lambda>0$ there exists a function $w:X \to \mbb R$ such that both $\mcl M$ and $\phi_1(\mcl M, \lambda)$ satisfy Property \ref{ass: terminal} using $w$.
	
\end{rem}

In the next lemma we will give the Lipschitz properties of the OCTGF of a $\lambda$-smoothed MDP.
\begin{lem} \label{lem: cost to go has bounded w norm}
	For some $\lambda>0$ consider the OCTGF's $J_t$ and $\tilde{J}_t$ of the MDP's $\mathcal{M}
	\in \mathbb{A}$ and  $\tilde{\mathcal{M}}
	= \Phi_1(\mathcal{M}, \lambda)$ respectively. Then $||J_t||_{w}<\infty$ and $\tilde{J}_t \in \mathbb{L}_w(X)$, where $w:X \to \mathbb{R}$ is as in Property \ref{ass: terminal} for $\mathcal{M}$. Furthermore,
	\begin{align} \label{eqn: recursion eqaution that defines Lip constant of J tilde} L_{\tilde{J}_{ t}}=[L_c+L_q L_{\tilde{J}_{t+1}}] [1 + L_{\psi}]  \\
	 L_{\tilde{J}_{T}}= L_{\tilde{h}_\lambda} + L_h. \nonumber \end{align}
	
\end{lem}
\begin{proof}
See Lemma 2.5 in \cite{F_Dufour_2012}. \end{proof}

\begin{cor} \label{cor: explicit lipschitz constant for J tilde} Consider the OCTGF, $\tilde{J}_{t}$, of a MDP $\tilde{\mathcal{M}} = \Phi_1(\mathcal{M}, \lambda)$ for some $\mathcal{M} \in \mathbb{A}$. Then its Lipschitz constant, $L_{\tilde{J}_{ t}}>0$ satisfies,\\
\small{	$L_{\tilde{J}_{ t}}=(L_q[1 + L_{\psi}])^{T-t}[L_{f_{1}}+L_{f_{2}} + 2 \lambda m\max\{||f_1||_{\infty},||f_2||_{\infty}\} + L_h] + L_c[1+L_{\psi}] \sum_{i=1}^{T-t} (L_q[1 + L_{\psi}])^{i-1} \quad \forall t \in [T]$.}
\normalsize
\end{cor}

The next Proposition proves that the OCTGF for a $\lambda$-smoothed MDP converges to the OCTGF of its corresponding approximable MDP under the supremum norm as $\lambda \to \infty$.

\begin{prop} \label{prop: J and J tilde arbitrarily close}
	Consider an MDP $\mathcal{M}=((\{X_t\}_{t \in \mathbb{N}},X),\mathcal{A},\psi, \{Q_t\}_{t\in \mathbb{N}}, (c,h),T) \in \mathbb{A}$ and its corresponding $\lambda$-smoothed MDP $\tilde{\mathcal{M}}=((\{X_t\}_{t \in \mathbb{N}},X),\mathcal{A},\psi, \{Q_t\}_{t\in \mathbb{N}}, (c,h_1 + \tilde{h}_\lambda),T)=\Phi_1(\mathcal{M}, \lambda)$ with OCTGF's denoted by $J_{t}(x)$ and $\tilde{J}_{t}(x)$ respectively. Then for $\theta>0$ there exists $\Lambda \in \mathbb{R}$ and $w:X \to \mathbb{R}$ such that $||J_t||_{w}< \infty$, $\tilde{J}_t \in \mathbb{L}_w(X)$ and for all  $\lambda>\Lambda$ we have $\sup_{x \in X} \left|\tilde{J}_{t}(x)- J_{t}(x) \right|<  \left(||J_T||_w+ ||\tilde{J}_T||_w \right) \theta$ for any $t \in [T-1]$.
\end{prop}

\begin{proof}
	Consider $w:X \to \mathbb{R}$ as in Property \ref{ass: terminal} of $\mathcal{M}$ then by Lemma \ref{lem: cost to go has bounded w norm} $||J_t||_{w}< \infty$, $\tilde{J}_t \in \mathbb{L}_w(X)$.
	 For $t \in [T-1]$ using Bellman's equation \eqref{bellman} we have,
	\small{
	\begin{align} \label{eqn: initial J and J tilde inequality}
	&\left|\tilde{J}_{t-1}(x)- J_{t-1}(x)\right|  \\ \nonumber
	 & \le \left( \inf_{u \in \psi(x)} \sup_{a \in \psi(x)} \left|\int_{X_{T}} \tilde{J}_{t}(y)Q_{t-1}(dy|x,u) -  \int_{X_{T}} J_{t}(y)Q_{t-1}(dy|x,a) \right| \right. \\ \nonumber
	 & \left. \qquad \lor  \inf_{a \in \psi(x)} \sup_{u \in \psi(x)} \left|\int_{X_{T}} \tilde{J}_{t}(y)Q_{t-1}(dy|x,u) -  \int_{X_{T}} J_{t}(y)Q_{t-1}(dy|x,a) \right| \right) \\ \nonumber
	 & \quad + \left( \inf_{u \in \psi(x)} \sup_{a \in \psi(x)}\left|c(x,u)-c(x,a)\right| \lor \inf_{a \in \psi(x)} \sup_{u \in \psi(x)} \left|c(x,u)-c(x,a)\right| \right) \nonumber
	\end{align} }
\normalsize
We now proceed by downward induction starting at $t=T-1$.
Let $\theta>0$, by Property \ref{ass: proability transitions to region tends to 0} of $\mathcal{M}$ $\exists \Lambda>0$ such that $\forall \lambda > \Lambda$
\vspace{-0.15cm}
\small
\begin{align*}
\left|sup_{x \in X,u \in \psi(x)} \int_{y \in \mathcal{R}_{\lambda,B}} w(y) Q_{T-1}(dy|x,u)\right|< \theta
\end{align*}
\normalsize
For $\lambda > \Lambda$ we see,
\small
\begin{align} \label{eqn: intial expected J and J tilde inequality}
 & \left|\int_{X_T} \tilde{J}_{T}(y)Q_{T-1}(dy|x,u) -  \int_{X_T} J_{T}(y)Q_{T-1}(dy|x,a) \right|\\ \nonumber
 & \le  \left|\zeta_{h_1+\tilde{h}_{\lambda},T}^{\mathcal{M}}(x,u) - \zeta_{h_1+\tilde{h}_{\lambda},T}^{\mathcal{M}}(x,a)\right|\\ \nonumber
 & \qquad +\left|\int_{X_T} \tilde{J}_{T}(y)Q_{T-1}(dy|x,a)  -  \int_{X_T} J_{T}(y)Q_{T-1}(dy|x,a)\right|\\ \nonumber
 & \le L_{q}(L_{h_1}+L_{\tilde{h}_{\lambda}})d_A(u,a)\\ \nonumber
 & \quad  + \left|\int_{X_T} \tilde{J}_{T}(y)Q_{T-1}(dy|x,a)  -  \int_{X_T} J_{T}(y)Q_{T-1}(dy|x,a)\right| \\ \nonumber
 & \le L_{q}(L_{h_1}+L_{\tilde{h}_{\lambda}})d_A(u,a) + \left|\int_{y \in \mathcal{R}_{\lambda,b}} \tilde{J}_{T}(y)Q_{T-1}(dy|x,a)\right|  \\ \nonumber
 & \hspace{0.15cm} +  \left| \int_{y \in \mathcal{R}_{\lambda,b}} J_{T}(y)Q_{T-1}(dy|x,a)\right| \\ \nonumber
 & \hspace{0.15cm} + \left|\int_{y \in X_T / \mathcal{R}_{\lambda,b}} \tilde{J}_{T}(y)Q_{T-1}(dy|x,a) -  \int_{y \in {X_T} / \mathcal{R}_{\lambda,b}} J_{T}(y)Q_{T-1}(dy|x,a) \right| \\ \nonumber
 & \le L_{q}(L_{h_1}+L_{\tilde{h}_{\lambda}})d_A(u,a)\\ \nonumber
 & \quad + (||J_T||_w+ ||\tilde{J}_T||_w ) \left|\int_{y \in \mathcal{R}_{\lambda,b}}w(y) Q_{T-1}(dy|x,a)\right|\\ \nonumber
 & \quad  + \left|\int_{y \in X_T / \mathcal{R}_{\lambda,b}} [(h_1(y) + \tilde{h}_\lambda(y))- h(y)]Q_{T-1}(dy|x,a)\right|\\ \nonumber
  & \le L_{q}(L_{h_1}+L_{\tilde{h}_{\lambda}})d_A(u,a) + (||J_T||_w+ ||\tilde{J}_T||_w ) \theta.  \\  \nonumber    
\end{align}

 \vspace{-0.7cm} \normalsize{Where the triangle inequality is used in the first and third inequality, Property \ref{property: Zeta is lipshitz} is used in the second inequality, Property \ref{ass: terminal} is used in the fourth inequality and in the fifth inequality Property \ref{ass: proability transitions to region tends to 0} and the fact $h(y)=h_1(y) + \tilde{h}_{\lambda}(y)$ $\forall y \in X/ \mcl R_{\lambda,b}$ is used.}\\
Also by Property \ref{ass: terminal} of $\mcl M$, \begin{equation} \label{eqn: cost inequality} |c(x,u)-c(x,a)|\le L_c d_A(u,a). \end{equation}
Moreover, \begin{align} \label{eqn: d(_A(u,a) = 0)} \inf_{u \in \psi(x)} \sup_{a \in \psi(x)} d_A(u,a) \lor \inf_{a \in \psi(x)} \sup_{u \in \psi(x)} d_A(u,a) \\
= d_H(\psi(x),\psi(x))=0 \nonumber .\end{align}
Thus it follows by substituting $t = T$ into \eqref{eqn: initial J and J tilde inequality} and further using \eqref{eqn: intial expected J and J tilde inequality}, \eqref{eqn: cost inequality} and \eqref{eqn: d(_A(u,a) = 0)},
\begin{equation*}
\sup_{x \in X}\left|\tilde{J}_{\lambda,T-1}(x)- J_{T-1}(x)\right| \le  (||J_T||_w+ ||\tilde{J}_T||_w ) \theta.
\end{equation*}
Now we proceed by downward induction. Assuming the result to be true for $s+1$, $\exists \Lambda$ such that $\forall \lambda> \Lambda$ we have $\sup_{x \in X} \left|\tilde{J}_{s+1}(x)- J_{s+1}(x) \right|<  \left(||J_T||_w+ ||\tilde{J}_T||_w \right) \theta$. Now for $\lambda> \Lambda$,

%
\small{
	\begin{align} \label{eqn: 2 intial expected J and J tilde inequality}
	& \left|\int_{X_{s+1}} \tilde{J}_{s+1}(y)Q_{s}(dy|x,u) -  \int_{X_{s+1}} J_{s+1}(y)Q_{s}(dy|x,a)\right|\\ \nonumber
	& \le \left|\zeta_{\tilde{J}_{s+1},s}^{\mathcal{M}}(x,u)-\zeta_{\tilde{J}_{s+1},s}^{\mathcal{M}}(x,a)\right| + \int_{X_{s+1}} \left|\tilde{J}_{s+1}(y)- J_{s+1}(y)\right|Q_{s}(dy|x,a) \\ \nonumber
	& \le L_{P}L_{\tilde{J}_{s+1}}d_A(u,a) +  \sup_{y \in X}|\tilde{J}_{s+1}(y)- J_{s+1}(y)| \int_{X_{s+1}} Q_{s}(dy|x,a)\\ \nonumber
	& \le L_{P}L_{\tilde{J}_{s+1}}d_A(u,a) + (||J_T||_w+ ||\tilde{J}_T||_w ) \theta. \nonumber
	\end{align}
}
\normalsize
Where the first inequality uses the triangle rule, the second inequality uses Property \ref{property: Zeta is lipshitz} and the third inequality uses the induction hypothesis.\\
Thus it follows by substituting $t = s+1$ into \eqref{eqn: initial J and J tilde inequality} and further using \eqref{eqn: 2 intial expected J and J tilde inequality}, \eqref{eqn: cost inequality} and \eqref{eqn: d(_A(u,a) = 0)},
\begin{align*}
 |\tilde{J}_{s}(x)- J_{s}(x)|
 & \le (||J_T||_w+ ||\tilde{J}_T||_w ) \theta. 
\end{align*} \end{proof}

\subsection{Discretization}
In this section we show how to mathematically discretize a $\lambda$-smoothed readily-aporximable MDP. We will show the discretized MDP can be made arbitrarily close to the $\lambda$-smoothed MDP.

\begin{defn}
		For any MDP $\tilde{\mathcal{M}}=((\{X_t\}_{t \in \mathbb{N}},X),\mathcal{A},\psi, \{Q_t\}_{t\in \mathbb{N}}, (c,h),T) \in \Phi_1(\mathbb{A}, \mathbb{R}^+)$ and $\epsilon>0$ we define $\mathbb{H}_{\tilde{\mathcal{M}},\alpha }$ as a set of compact subsets of $X$  where we say $\{H_t\}_{t \in [T]}\in \mathbb{H}_{\tilde{\mathcal{M}}, \alpha}$ if
		\vspace{-0.05cm}
		\begin{itemize}
			\item $H_t$ is a compact subset of $X_t$ for all $t \in [T]$ and
			\item $\sup_{(x,a) \in \mathbb{J}_{t-1}} \int_{X/H_{t}} w(y) Q_{t-1}(dy|x,a) \le \alpha \text{ for all } t \in [T]  $, where $\mbb J_{t-1} = \{(x,a): x \in H_{t-1}, a \in \psi(x)\}$ and $w:X \to \mathbb{R}$ is as in Property \ref{ass: terminal} for some $\mathcal{M}$ such that $\tilde{\mathcal{M}}=\Phi_1(\mcl M,\lambda)$ for some $\lambda >0$.
		\end{itemize}
		
\end{defn}
\begin{lem} \label{eps_def}
	$\forall \tilde{\mathcal{M}} \in \Phi_1(\mathbb{A}, \mathbb{R}^+)$ and $\alpha>0$ $\mathbb{H}_{\tilde{\mathcal{M}}, \alpha} \ne \emptyset$.
\end{lem}
\begin{proof}
	We note $\Phi_1(\mathbb{A}, \mathbb{R}^+)$ is a subset of $\mathbb{A}$ and thus the result of the Lemma follows from Lemma 2.9 \cite{F_Dufour_2012}.
\end{proof}
Consider $\alpha>0$. For some $\{H_t\}_{t \in [T]} \in \mathbb{H}_{\tilde{\mathcal{M}}, \alpha}$ let us denote the map $\Phi_2: X^\cup \times \Phi_1(\mathbb{A}, \mathbb{R}^+) \to \mathbb{M}$ such that for $\tilde{\mathcal{M}}=((\{X_t\}_{t \in \mathbb{N}},X),\mathcal{A},\psi, \{Q_t\}_{t\in \mathbb{N}}, (c,h),T) \in \Phi_1(\mathbb{A}, \mathbb{R}^+)$ we have $\Phi_2(\{H_t\}_{t \in [T]},\tilde{\mathcal{M}} )=((\{H_t\}_{t \in [T]},X),\mathcal{A},\psi, \{Q_t\}_{t\in \mathbb{N}}, (c,h),T)$.

\begin{prop} \label{prop:K_t lemma}
	Denote the MDP with tuple elements \eqref{1} to \eqref{6} by $\mathcal{M}$. For any $\lambda>1$ let $\alpha=\frac{\sqrt{2}m \max|a_{i,t}|}{\sqrt{ \log(\lambda)}} (\frac{1}{\lambda})^{\frac{1}{\max a_{i,t}^2}}$ and consider the family of sets $H_t:=\gamma_{t, \lambda}^m \subset \mathbb{R}^m$ where $\gamma_{t, \lambda}=[-t\sqrt{2\log(\lambda)},t\sqrt{2\log(\lambda)}]$, then
	$\{H_t\}_{t \in [T] } \in \mathbb{H}_{\phi_1(\mcl M, \lambda), \alpha  }$.
	
\end{prop}
\begin{proof}
	Clearly $H_t$ is a compact subset of $\R^m$ $\forall t \in [T]$. Next using $w(x) \equiv 1$ that can be used in Property \ref{ass: terminal} of $\mcl M$ we show,
	\small{ \begin{equation} \label{eqn: prop 5 result}
	\sup_{x \in H_{t-1} } \int_{X/H_{t}}  Q_{t-1}(dy|x) < \alpha.
	\end{equation} } \normalsize
	Recalling $\epsilon \sim \mcl N(0,1)$,	
	\small{	\begin{align} \label{prop 5: initial inequality}
	\sup_{x \in H_{t-1} } \int_{X/H_{t}} & Q_{t-1}(dy|x) = \sup_{x \in H_{t-1} }  \mbb{P}_{\epsilon}(x + A_{\cdot, t} \epsilon \notin H_t)\\ \nonumber
	& \le \sup_{x \in H_{t-1} } \bbl\{ \mathbb{P}_{\epsilon} \left(\cup_{i \in \{1,..,m\}} \left\{  x_i + a_{i,t} \epsilon > t\sqrt{2\log(\lambda)} \right\} \right)\\ \nonumber
	& \qquad  + \mathbb{P}_{\epsilon} \left(\cup_{i \in \{1,..,m\}} \left\{  x_i + a_{i,t} \epsilon < -t\sqrt{2\log(\lambda)} \right\} \right) \bbr\} \\ \nonumber
	& \le \sum_{i=1}^{m} \bbbl\{  \sup_{x_i \in \gamma_{t-1, \lambda} } \mathbb{P}_{\epsilon} \left( x_i + a_{i,t} \epsilon > t\sqrt{2\log(\lambda)}  \right)\\ \nonumber
	& \qquad   + \sup_{x_i \in \gamma_{t-1, \lambda} } \mathbb{P}_{\epsilon} \left(  x_i + a_{i,t} \epsilon < -t\sqrt{2\log(\lambda)} \right) \bbbr\}  \nonumber
	\end{align}}
\normalsize

	 We will show,  \begin{equation}
	 \label{eqn: prop 5 prob bounds 1} \sup_{x \in \gamma_{t-1, \lambda} } \mathbb{P}_{\epsilon} \left(x+a\epsilon >t\sqrt{2\log(\lambda)} \right)\le \frac{|a|}{\sqrt{2 \log(\lambda)}} \left(\frac{1}{\lambda}\right)^{\frac{1}{a^2}}
	 \end{equation} by considering the cases $a>0$, $a<0$ and $a=0$ separately. For  $a>0$,
	\small{	
	\begin{align} \nonumber
	  & \sup_{x \in \gamma_{t-1, \lambda} }  \mathbb{P}_{\epsilon} \left(x+a\epsilon >t\sqrt{2\log(\lambda)} \right) = \sup_{x \in \gamma_{t-1, \lambda} }  \mathbb{P}_{\epsilon} \left(\epsilon >\frac{t\sqrt{2\log(\lambda)}- x}{a} \right)\\ \nonumber
	& \le \mathbb{P}_{\epsilon} \left(\epsilon >\frac{\sqrt{2\log(\lambda)}}{a} \right)
	 < \frac{a}{\sqrt{2\log(\lambda)}}\exp \left(-\frac{\log \left(\lambda \right) }{a^2}\right) \\ \nonumber
	& \le \frac{a}{\sqrt{2 \log(\lambda)}} \left(\frac{1}{\lambda}\right)^{\frac{1}{a^2}}. \nonumber
	\end{align} } \normalsize
	Where the second inequality uses Lemma \ref{lemma:probability inequality} and the last inequality follows since we can assume $\lambda>1$.
The case $a<0$ follows by a similar proof.
The case $a=0$ is trivial, $\sup_{x \in \gamma_{t-1, \lambda} }\mathbb{P}(x+a\epsilon >t\sqrt{2\log(\lambda)})=0$. A similar argument of considering the different cases of $a$ can show
\begin{equation} \label{eqn: prop 5 prob bounds 2} \sup_{x \in \gamma_{t-1, \lambda} } \mathbb{P}(x+a\epsilon <-t\sqrt{2\log(\lambda)})\le \frac{|a|}{\sqrt{2 \log(\lambda)}} \left(\frac{1}{\lambda}\right)^{\frac{1}{a^2}}.
\end{equation} Now, substituting~\eqref{eqn: prop 5 prob bounds 1} and~\eqref{eqn: prop 5 prob bounds 2} into \eqref{prop 5: initial inequality}, we get~\eqref{eqn: prop 5 result}. \end{proof}
Next we return to a general $\lambda$-smoothed MDP and approximate its state and control space's with a countable set.


\begin{defn} \label{Disc_MDP_def}
Given an approximable MDP $\mathcal{M} \in \mathbb{A}$ we can define a corresponding MDP $\tilde{\mathcal{M}}= ((\{H_t\}_{t \in [T]},X),\mathcal{A},\psi, \{Q_t\}_{t\in \mathbb{N}}, (c,\tilde{h}),T) = \Phi_2(\{H_t\}_{t \in [T]}, \Phi_1(\mathcal{M}, \lambda))$ for some compact family $\{H_t\}_{t \in [T]}$ and $\lambda>0$. Furthermore given
	\begin{itemize}
		\item  $\Gamma_t$ is a $\beta$-partition of $H_t$.
		\item $\Theta_t(x)$ is a $\eta$-partition of $\psi(x)$.
		\item $\hat Q$ defined for $x \in \Gamma_t$, $y \in \Gamma_{t+1}$ and $a \in \Theta_t(x)$ as $\hat{Q}_t(y|x,a):=Q_t(p_{H_{t+1},\Gamma_{t+1}}^{-1}(y)|x,a)$ for $t\in [T-1]$ where $p_{H_{t+1},\Gamma_{t+1}}^{-1}(y)$ is the pre-image of $y$,
	\end{itemize}
we define the map $\Phi_3: \mathbb{R}^+ \times \mathbb{R}^+ \times Im\{\Phi_2\} \to \mathbb{M}$ by $\hat{\mathcal{M}}=\Phi_3(\beta, \eta, \tilde{\mathcal{M}})$ if $\hat{\mathcal{M}}=((\{\Gamma_t\}_{t \in [T]},X), \mathcal{A},\Theta_t, \hat{Q}, (c,\tilde{h}),T)$.

	Note if $\psi(x)=\emptyset$ $\forall x \in X$ we simplify our notation and set $\eta=0$.
	\end{defn}
	\begin{defn}
		We define the total approximation map $\Psi: \mathbb{A} \times (\mathbb{R}^{+})^3 \times ( X^\cup)^T  \to \mathbb{M}$ by $\Psi(\mathcal{M},\lambda, \beta, \eta, \{H_t\}_{t \in [T]} ) = \Phi_3(\beta, \eta,  \Phi_2(\{H_t\}_{t \in [T]},\Phi_1(\mathcal{M}, \lambda)))$.
	\end{defn}
	\begin{defn}
		Let $\hat{J}_{t}$ be the OCTGF of some MDP in $Im\{\Psi\}$ with state space $(\{\Gamma_t\}_{t \in [T]} ,X)$. For any compact set $H_t$ such that $\Gamma_t \subset H_t$ we define the extended OCTGF of $\hat{J}_{t}(x)$ as follows:
		\begin{equation*}
		\hat{F}_{t}(x)= \hat{J}_{t}(p_{H_t,\Gamma_t}(x)) \quad x \in H_t
		\end{equation*}
		


\end{defn}
Next we will state a theorem that gives a bound for error of the OCTGF's of the MDP's in $Im\{\Phi_1\}$ and associated MDP's mapped under $\Phi_3$.

\begin{thm} \label{prop: error bound on K_t optimal cost to go functions}
		Consider some $\mathcal{M} \in \mathbb{A}$. For some $\lambda >0$ suppose $\tilde{J}_t$ is the OCTGF of $\tilde{\mathcal{M}} = \Phi_1(\mathcal{M}, \lambda)$. For any $\alpha, \beta, \eta>0$ there exists $\{H_t\}_{t \in [T]} \in \mbb H_{\tilde{\mcl M}, \alpha }$ such that if we denote $\hat{J}_t$ as the extended OCTGF of  $\Psi(\mathcal{M},\lambda, \beta, \eta, \{H_t\}_{t \in [T]} )$ then for $t \in [T-1]$,
	\begin{equation} \label{eqn: bound for J tilde and J hat 1}
	\sup_{x \in H_T}|\tilde{J}_{T}(x)-\hat{J}_{T}(x)| \le (L_h+L_{\tilde{h}_{\lambda}}) \beta
	\end{equation}
	
	\vspace{-0.3cm}
	\begin{align} \label{eqn: bound for J tilde and J hat 2}
	\sup_{x \in H_t}|\tilde{J}_{t}(x)-\hat{J}_{t}(x)| \le & ||\tilde{J}_{t+1}||_w \alpha + (L_{\tilde{J}_{t+1}}L_{q} +L_c) \eta\\ \nonumber
	& + \sup_{y \in H_{t+1}}|\tilde{J}_{t+1}(y)- \hat{J}_{t+1}(y)| + L_{\tilde{J}_{t}} \beta \nonumber
	\end{align}
Where the function $w:X \to \mathbb{R}$ and Lipschitz constant $L_q$ are as in Property \ref{ass: terminal} and Property \ref{property: Zeta is lipshitz} of $\mathcal{M}$ respectively.

Moreover in the case where the control space of $\mcl M$ is empty we set $\eta=0$.

\end{thm}

\begin{proof} See Theorem 3.4 in \cite{F_Dufour_2012}.
\end{proof}

%

\subsection{Error Bounds}
In this section we will show how Theorem \ref{prop: error bound on K_t optimal cost to go functions} can be combined with Proposition \ref{prop: J and J tilde arbitrarily close} to show that the OCTGF's of an MDP $\mathcal{M} \in \mathbb{A}$, and the approximated MDP $\hat{\mathcal{M}}= \Psi(\mathcal{M},\lambda, \beta, \eta, \{H_t\}_{t \in [T]} )$ are arbitrary close together.

\begin{thm} \label{theorem: error bound for J and J hat}
		Consider some MDP $\mathcal{M}
		\in \mathbb{A}$ with OCTGF denoted by $J_t$. For any $\alpha, \beta, \eta, \theta>0$ there exists $\{H_t\}_{t \in [T]} \in \mbb H_{\tilde{M},\alpha}$ and $\Lambda>0$ such that for all $\lambda>\Lambda$ and any $t \in [T-1]$,

		\vspace{-0.2cm}
		\small{
	\begin{align} \label{eqn: main bound}
	&\sup_{x \in H_t}|J_{t}(x)-\hat{J}_{t}(x)|
	\le \theta \left\{||J_T||_w+ ||\tilde{J}_{T}||_w \right\}    + \alpha \left\{ \sum_{i=1}^{T-t}||\tilde{J}_{t+i}||_w \right\} \\ \nonumber
	& +  \eta \left\{ L_{q}\left(\sum_{i=1}^{T-t}L_{\tilde{J}_{t+i}} +(T-t)L_c\right) \right\} + \beta \left\{ L_h+L_{\tilde{h}_{\lambda}} + \sum_{i=1}^{T-t}L_{\tilde{J}_{t+i -1}} \right\} \nonumber
	\end{align}}
\normalsize
	Where $\tilde{J}_{t}$ and $h +\tilde{h}_{\lambda}$ is the OCTGF and terminal cost function of $\Phi_1(\mathcal{M}, \lambda)$ respectively. $\hat{J}_{t}$ is the extended OCTGF of $\Psi(\mathcal{M},\lambda, \beta, \eta, \{H_t\}_{t \in [T]} )$. The function $w:X \to \mathbb{R}$ and constant $L_q$ are as in Properties \ref{ass: terminal} and \ref{property: Zeta is lipshitz} of $\mathcal{M}$ respectively.
	
Moreover in the case where the control space of $\mcl M$ is empty we set $\eta=0$.	

\end{thm}

\begin{proof}
	By the triangle inequality,
	\begin{align} \label{eqn: initial inequality for J and J hat}
	|J_{t}(x)-\hat{J}_{t}(x)| \le |J_{t}(x)-\tilde{J}_{t}(x)|+|\tilde{J}_{t}(x)-\hat{J}_{t}(x)|
	\end{align}
	We then use Proposition \ref{prop: J and J tilde arbitrarily close} to bound $|J_{t}(x)-\tilde{J}_{t}(x)|$. Then we recursively solve \eqref{eqn: bound for J tilde and J hat 1} and \eqref{eqn: bound for J tilde and J hat 2} in Theorem \ref{prop: J and J tilde arbitrarily close} to bound $|\tilde{J}_{t}(x)-\hat{J}_{t}(x)|$. Substituting these bounds into \eqref{eqn: initial inequality for J and J hat} the result~\eqref{eqn: main bound} follows.
\end{proof}

 We now specialize Theorem \ref{theorem: error bound for J and J hat} to the MDP with tuple elements defined \eqref{1} to \eqref{6}. Since in this specific case the control space, $\mcl A$, is empty we can set $\eta=0$ in Theorem \ref{theorem: error bound for J and J hat}.


\begin{cor} \label{cor:error bound for polytope DP}
	Consider the MDP with tuple elements defined \eqref{1} to \eqref{6} by $\mathcal{M}$
	. For any $\lambda> 1$ and $ \beta>0$ let  $H_t=[-t\sqrt{2\log(\lambda)},t\sqrt{2\log(\lambda)}]^m$ then the extended OCTGF, $\hat{J}_t$, of $\Psi(\mathcal{M},\lambda, \beta, 0, \{H_t\}_{t \in [T]} )$ satisfies,
		\begin{align} \label{eqn: error bound for poltopic integration}
		& \left|\int_{{x} \in \mathcal{P}}\phi( {x}) d {x} -    \hat{J}_{0}(x_0) \right|
		  \le \frac{2m}{\min_{i\in[m]} \{|a_{i,T}|\} \lambda}\\ \nonumber
		  & \hspace{1.7cm} + \frac{\sqrt{2}m T \max|a_{i,t}|}{\sqrt{ \log(\lambda)}} (\frac{1}{\lambda})^{\frac{1}{\max a_{i,t}^2}} + 2  m \lambda \beta (T+1).\nonumber
		\end{align}
	where $x_0=(0,..,0)$, $\mathcal{P}=\{ {x} \in \mathbb{R}^T : A  {x} < b\}$ and $A \in \R^{m \times T}$ and $b \in \R^m$.
\end{cor}
\begin{proof}
	By Proposition \ref{prop: proves our polytope MDP is readily-aproximable} $\mathcal{M} \in \mathbb{A}$, thus $\hat{\mathcal{M}}$ is well defined. Let us denote the OCTGF of the MDP's $\mathcal{M}$ and $\mcl{\tilde{M}}=\Phi_1(\mcl M, \lambda)$ by $J_{t}$ and $\tilde{J}_t$. By Proposition \ref{prop: MDP = polytope integral} $J_0(x_0)=\int_{\mathbf{x} \in \mathcal{P}}\phi(\mathbf{x}) d\mathbf{x}$. Using Corollary \ref{cor: explicit lipschitz constant for J tilde} and Corollary \ref{cor: h tilde Lipschitz.} $L_{\tilde{J}_t}$ and $L_{\tilde{h}}$ can be calculated. Proposition \ref{prop:K_t lemma} shows $H_t \in \mbb{H}_{\mcl{\tilde{M}}, \alpha}$ where $\alpha=\frac{\sqrt{2}m \max|a_{i,t}|}{\sqrt{ \log(\lambda)}} (\frac{1}{\lambda})^{\frac{1}{\max a_{i,t}^2}}$. Now Theorem \ref{theorem: error bound for J and J hat} can be applied to the specific MDP $\mathcal{M}$, where Lemma \ref{lem: Bound_of_transition_to_bad_region} is used to select an appropriate $\theta$; and using induction and \eqref{bellman} it can be shown $||\tilde{J}_t||_w \le 1$ $\forall t \in [T]$. \end{proof}

\section{Numerical Results} \label{Section 5: Numerical Results}


Using the smoothing and discretization procedure laid out in section \ref{section 4: Approximating MDP's} we find approximate solutions to the optimization problem associated with the MDP with tuple elements defined \eqref{1} to \eqref{6}. The algorithm recursively solves \eqref{bellman} for different discretization parameters, $(\lambda, \beta)$. In all simulations $(\lambda, \beta)$ were parametrized by $n \in \mbb N$; we selected $\lambda=\sqrt{n}$ and $\beta= \frac{1}{n}$. Figure \ref{fig: error} shows the results of computing the probability that a two dimensional Gaussian variable is in the positive orthant; this can be written as an integral of the form \eqref{eqn: polytope integral} where $A =[1, 1]^T$ and $b=[0,0]^T$. Using \eqref{eqn: error bound for poltopic integration} it can be shown the error bound for integration over the positive orthant is $E=\frac{28}{\sqrt{n}} + \frac{4}{ \sqrt{n\log{n}}}$, which is of order $\mcl O \left(\frac{1}{\sqrt{n}} \right)$. The order of the actual error of the algorithm, when compared to the true value of 0.5, seems to also be $\mcl O \left(\frac{1}{\sqrt{n}} \right)$; indicating our error bounds are tight in some cases.

In Figure \ref{fig: A_0_5__0_7__1_5__0_9___0_2__0_7__0_5_1_3_B_2___0_5} we evaluate an integral of the Form \eqref{eqn: polytope integral} where $A=\begin{bmatrix}
& 0.5 & 0.7 & 1 & 0.9\\
& 0.2 & 0.7 & 0.5 & 1
\end{bmatrix}$ and $b= \begin{bmatrix}
& 2\\
& 0.5
\end{bmatrix}$. The horizontal line represents the Monte Carlo approximation of $10^8$ samples. The curved line represents the OCTGF, $\hat{J}_{0}(0)$, of the approximated MDP with tuple elements \eqref{1} to \eqref{6}.

\begin{figure}[h]
	\centering
	\includegraphics[width=0.5\textwidth]{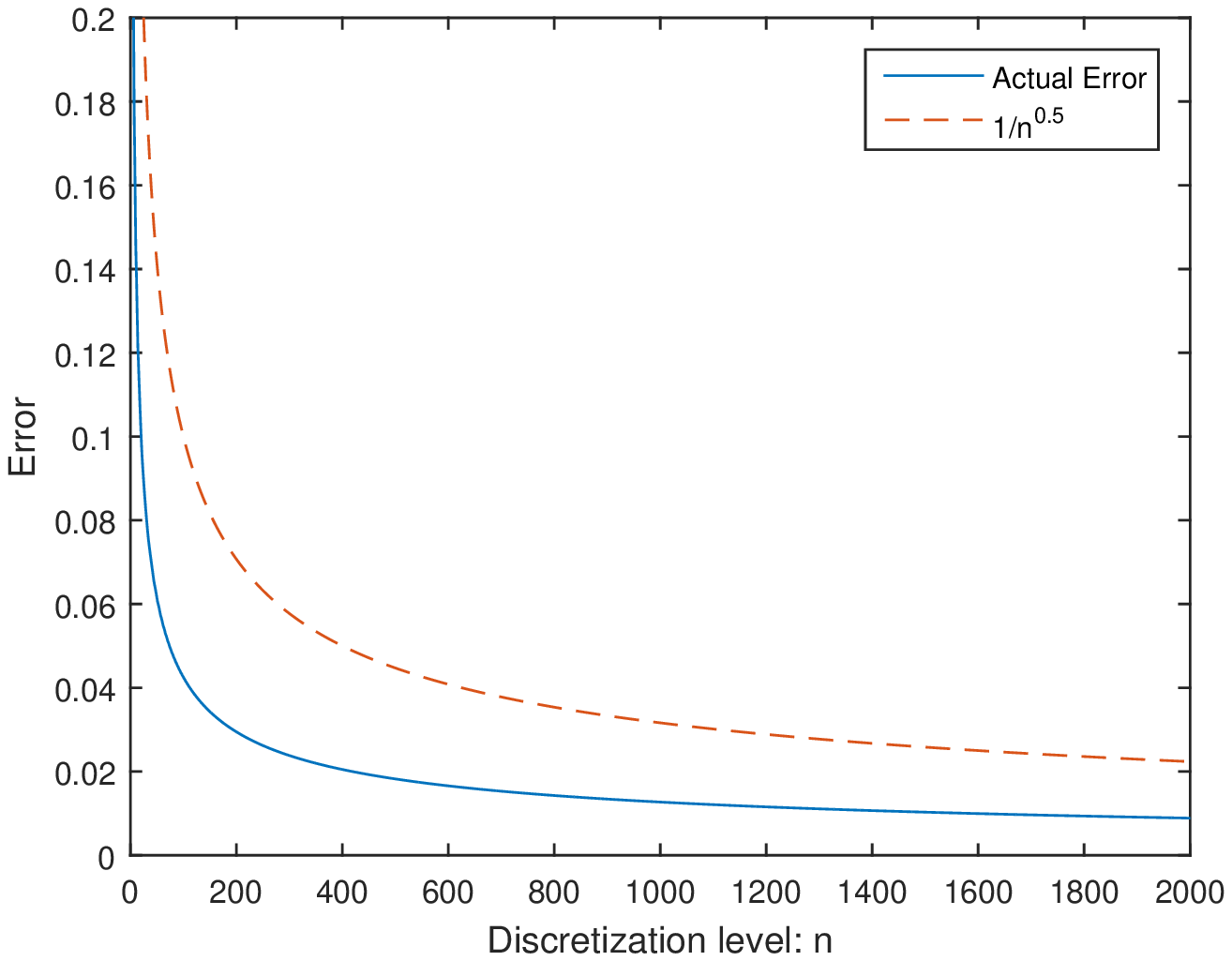}
	\caption{}
	\label{fig: error}
\end{figure}

\begin{figure}[h]
	\centering
	\includegraphics[width=0.5\textwidth]{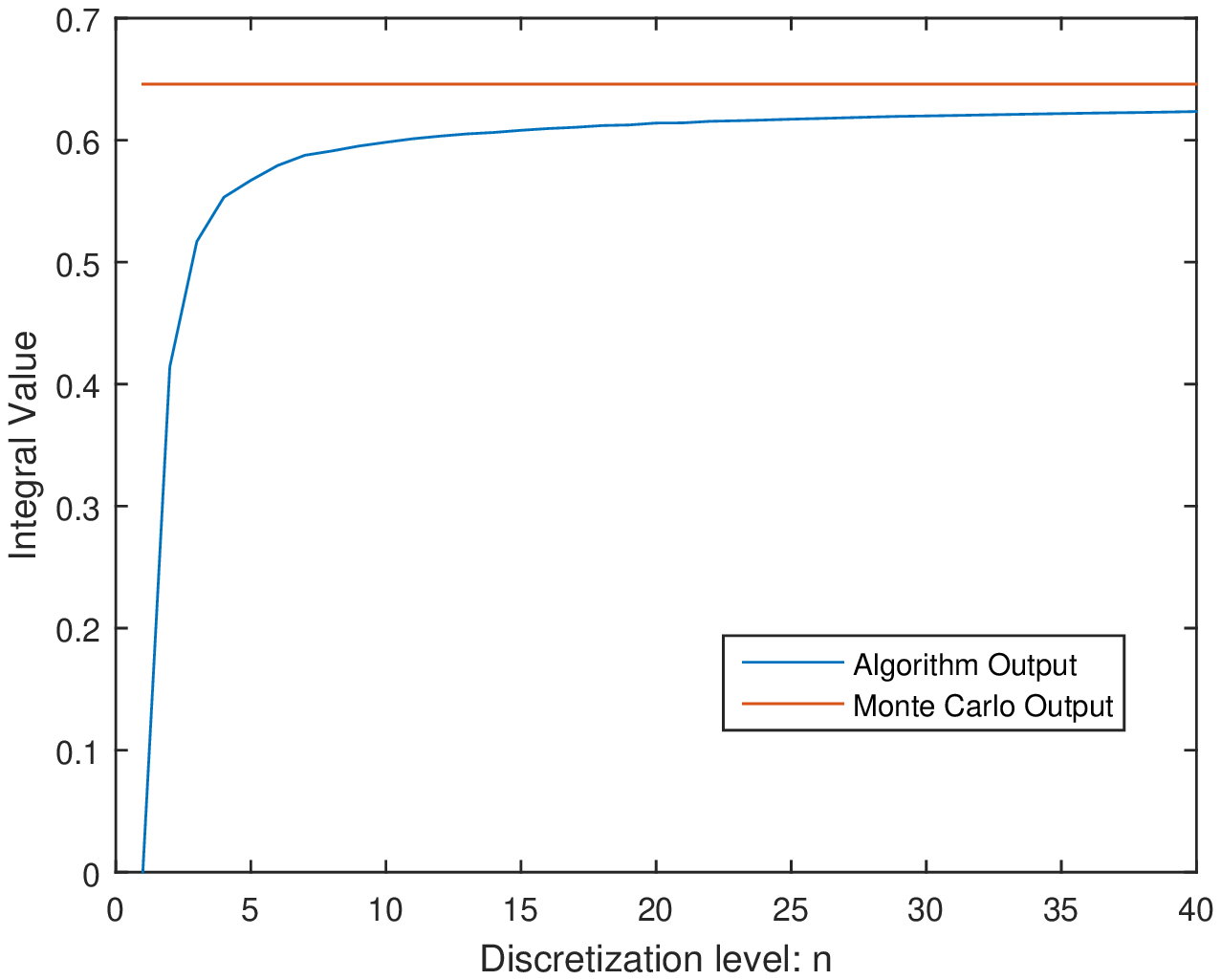}
	\caption{}
	\label{fig: A_0_5__0_7__1_5__0_9___0_2__0_7__0_5_1_3_B_2___0_5}
\end{figure}

\section{Conclusion} \label{conclusion}

In this paper we showed that given a multivariate Gaussian integral over a polytope it is possible to construct an MDP such that the solution of the MDP's associated optimization problem is equal to the integral. In general this class of MDP's have non-compact uncountable state spaces and discontinuous terminal cost functions. In this paper we use Bellman's equation to solve the associated optimization problem. However in general there is no analytical solution to Bellman's equation for MDP's of this class and thus an approximation is required. We proposed an approximation scheme that maps our class of MDP's to a much simpler class of MDP's with countable state and control spaces. Moreover we derived bounds on the supremum norm error of the optimal cost to go functions of the MDP and the mapped MDP. The main contribution of this paper is thus a dynamic programing based algorithm for evaluating multivariate Gaussian integration over polytopes with a priori error bounds.

Our numerical results presented in section \ref{Section 5: Numerical Results} are consistent our error bounds in section \ref{section 4: Approximating MDP's}. There are substantial computational costs to this dynamic programing approach but using this approach we are able to compute the integral to any degree of accuracy. This paper links computing multivariate Gaussian integration over polytopes to dynamic programing; a well developed computational technique. 
\section{Appendix}

\begin{lem} \label{lemma:probability inequality}
	For $t >0$ $\mathbb{P}_{\epsilon}(\epsilon>t) \le \frac{1}{t} \exp(\frac{-t^2}{2})$ where $\epsilon \sim \mathcal{N}(0,1)$.
\end{lem}
\begin{proof}
	\vspace{-0.8cm}
	\begin{align*}
	\mathbb{P}_{\epsilon} \left(\epsilon>t \right)= & \int_{t}^{\infty}\frac{1}{\sqrt{2 \pi}} \exp \left(\frac{-x^2}{2} \right) dx\\
	< & \int_{t}^{\infty} \frac{x}{t}  \exp \left(\frac{-x^2}{2} \right) dx \\
	= & \frac{1}{t} \exp \left(\frac{-t^2}{2} \right)
	\end{align*}
	Where the second inequality uses the fact that inside the integral domain $x \ge t$ and $\frac{1}{\sqrt{2 \pi}}<1$.
\end{proof}

\bibliographystyle{ieeetr}
\bibliography{bibliography_Extension_non_lip}

\end{document}